\documentclass[a4paper,11pt]{amsart}
\usepackage{amsmath,amsthm,amssymb}
\usepackage[arrow,matrix]{xy}
\usepackage{hyperref}

\theoremstyle{plain}
\newtheorem{theorem}{Theorem}[section]

\newtheorem{lemma}[theorem]{Lemma}

\newtheorem{corollary}[theorem]{Corollary}

\theoremstyle{definition}

\newcommand{\Z}{\mathbb{Z}}

\newcommand{\C}{\mathbb{C}}
\newcommand{\CP}{\mathbb{C}P}

\newcommand{\mathscr}{}

\def\w#1{w_{#1}}

\begin{document}%
\title{Strong cohomological rigidity of a product of projective spaces}
\author {Suyoung Choi}
\address{Department of Mathematics, Osaka City University, Sugimoto, Sumiyoshi-ku, Osaka 558-8585, Japan}
\email{choi@sci.osaka-cu.ac.jp}
\urladdr{http://math01.sci.osaka-cu.ac.jp/~choi}

\author{Dong Youp Suh}
\address{Department of Mathematical Sciences, KAIST, 335 Gwahangno, Yu-sung Gu, Daejeon 305-701, Korea}
\email{dysuh@math.kaist.ac.kr}

\thanks{The first author is supported by the Japanese Society for the Promotion of Sciences (JSPS grant no. P09023).}
\thanks{The second author is partially supported by Basic Science Research Program through the national Research Foundation of Korea(NRF) founded by the Ministry of Education, Science and Technology (2009-0063179).}

\subjclass[2000]{Primary 57S25; Secondary 22F30}
\keywords{product of projective spaces, strong cohomological rigidity, toric manifold, quasitoric manifold}

\date{\today}
\maketitle

\begin{abstract}
We prove that for a toric manifold $M$, any graded ring isomorphism $H^\ast(M) \to H^\ast(\prod_{i=1}^{m}\CP^{n_i})$ is induced by a diffeomorphism $\prod_{i=1}^m \CP^{n_i} \to M$.
\end{abstract}

\section{Introduction}
The \emph{cohomological rigidity problem} for toric manifolds asks whether the integral cohomology ring of a toric manifold determines its topological type or not. So far, there is no negative answer to the question but some positive results. In \cite{Ch-Ma-Su-2010}, the authors with M. Masuda show that if $M$ is a toric manifold whose cohomology ring is isomorphic to that of $\prod_{i=1}^m \CP^{n_i}$, a product of complex projective spaces, then $M$ is actually diffeomorphic to $\prod_{i=1}^m \CP^{n_i}$, which gives a positive result to the cohomological rigidity problem.

On the other hand, one might ask a stronger question as follows. Let $M$ and $N$ be two toric manifolds. Suppose there exists a graded ring isomorphism $\varphi \colon H^\ast(N \colon \Z) \to H^\ast(M \colon \Z)$. Then does there exist a diffeomorphism (or homeomorphism) $g \colon M \to N$ such that $g^\ast = \varphi$?

We call this question the \emph{strong cohomological rigidity problem}. This problem was actually suggested by the referee of the paper \cite{Ch-Ma-Su-2010}, which we could not answer until the paper \cite{Ch-Ma-Su-2010} was set for publication.

In this article, we would like to answer the strong cohomological rigidity problem for the case when $N=\prod_{i=1}^m \CP^{n_i}$. Namely, we prove the following theorem.

\begin{theorem}\label{theorem:rigidity of product of projective space}
Let $M$ be a toric manifold. If there is a graded ring isomorphism $\varphi : H^\ast(M) \rightarrow H^\ast(\prod_{i=1}^m \C P^{n_i})$, then there is a diffeomorphism $f : \prod_{i=1}^m \C P^{n_i} \rightarrow M$ such that $f^\ast=\varphi$.
\end{theorem}

Combining Theorem~\ref{theorem:rigidity of product of projective space} with \cite[Theorem 8.1]{Ch-Ma-Su-2010-2}, we obtain
the following corollary which generalizes \cite[Theorem 5.1]{Ma-Pa-2009} treating the case where $n_i=1$ for any $i$.
Note that a quasitoric manifold is a topological analogue of toric manifold, which was introduced by Davis and
 Januszkiewicz in \cite{DJ}.

\begin{corollary}
Let $M$ be a quasitoric manifold. If there is a graded ring isomorphism $\varphi : H^\ast(M) \rightarrow H^\ast(\prod_{i=1}^m \C P^{n_i})$, then there is a homeomorphism $f : \prod_{i=1}^m \C P^{n_i} \rightarrow M$ such that $f^\ast=\varphi$.
\end{corollary}

\section{Proof of Theorem \ref{theorem:rigidity of product of projective space}}
Let $R= \Z[x_1, \ldots, x_m] / <x_i^{n_i+1}\colon i=1, \ldots, m> \cong H^\ast(\prod_{i=1}^{m} \CP^{n_i})$.
\begin{lemma}\label{lemma:lemma0}
Let $y=\sum_{j=1}^m a_j x_j \in R$ such that $a_i \neq 0$ for some $i$. Then $y^{n_i} \neq 0$ in $R$.
\end{lemma}
\begin{proof}
Suppose $y^{n_i} =0$ on the contrary. Then $y^{n_i}$ must lie in the ideal generated by the polynomials $x_j^{n_j +1}$ for $j=1, \ldots, n$. Since $a_i \neq 0$, $y^{n_i} = (\sum_{j=1}^m a_j x_j)^{n_i}$ must contain the nonzero monomial term of $x_i^{n_i}$. However if a nonzero multiple of a power of $x_i$ appear in the ideal generated by $x_j^{n_j+1}$ for $j=1, \ldots, m$, then the exponent must be at least $n_i+1$, which is a contradiction.
\end{proof}

\begin{lemma}\label{lemma:lemma1}
If $\psi$ is a graded ring automorphism on $R$, then there exists a permutation $\sigma$ on $\{1, \ldots, m\}$ such that $n_i = n_{\sigma(i)}$ for all $i=1, \ldots, m$ and $\psi(x_i) = \pm x_{\sigma(i)}$.
\end{lemma}
\begin{proof}
Let $\psi(x_i) = \sum_{j=1}^{m} b_{ij}x_j$ for $i=1, \ldots, m$. Since $\psi$ is an automorphism, $\det B = \pm 1$, where $B = (b_{ij})$. Note that the positive integers $n_1,\ldots, n_m$ need not be distinct.
Let $S = \{N_1, \ldots, N_k \mid N_1 > \cdots > N_k\}$ be the set of distinct numbers from $n_1, \ldots, n_m$, and let $\mu \colon \{1, \ldots, m\} \to S$ be the function defined by $\mu(i)=n_i$. Let $J_\ell = \mu^{-1}(N_\ell)$ for $\ell = 1, \ldots, k$.

\textbf{Claim :} {\em $B$ is conjugate to a block upper triangular matrix by a permutation matrix.}

Since $x_i^{n_i+1}=0$ in $R$, $0=\psi(x_i^{n_i+1}) = (\sum_{j=1}^{m} b_{ij}x_j)^{n_i +1}$. Therefore, by Lemma~\ref{lemma:lemma0}, $b_{ij}=0$ if $n_i < n_j$. Hence by an appropriate permutation of the index set $\{1, \ldots, m\}$, we may assume that $n_1 \geq n_2 \geq \cdots \geq n_m$ and $B$ is an upper triangular matrix of the form
$$
    \left(
      \begin{array}{cccc}
        C_{J_1} &  &  & \ast  \\
         & C_{J_2} &  &  \\
         &  & \ddots &  \\
        0 &  &  & C_{J_k} \\
      \end{array}
    \right),
$$ where $C_{J_\ell}$ is the matrix formed from $b_{ij}$ with $i,j \in J_\ell$. This proves the claim.

Now let $J_{< \ell} = \bigcup_{\{N\in S\mid N<N_\ell\}} \mu^{-1}(N)$. By the previous claim, if $k \in J_\ell$, then we may write $\psi(x_k) = \sum_{j\in J_\ell} b_{kj}x_j + \sum_{j\in J_{<\ell}} b_{kj}x_j$. Let us denote $z_\ell = \sum_{j \in J_\ell} b_{kj}x_j$ and $\w{\ell} = \sum_{j \in J_{<\ell}}b_{kj}x_j$ for simplicity. Then $\psi(x_k) = z_\ell + \w{\ell}$. Therefore,
$$
    0 = \psi(x_k^{N_\ell+1}) = z_\ell^{N_\ell +1} + \binom{N_\ell +1}{1} \w{\ell}z_\ell^{N_\ell} + \binom{N_\ell +1}{2} \w{\ell}^2z_\ell^{N_\ell-1} + \cdots.
$$
We note that $z_\ell \neq 0$ since $\det B = \pm 1$. On the other hand, there is no way to get the polynomial equation
$$
    -\binom{N_\ell +1}{1} \w{\ell}z_\ell^{N_\ell} = z_\ell^{N_\ell +1} + \binom{N_\ell +1}{2} \w{\ell}^2z_\ell^{N_\ell-1} + \cdots
$$ in the ring $R$ unless $\w{\ell}=0$. Hence $z_\ell^{N_\ell+1} =0$. But then there is a unique nonzero $b_{ij}$ for $i \in J_\ell$, and hence $b_{ij} = \pm 1$.

Therefore, we have shown that $B$ is conjugate to a diagonal matrix all of whose diagonal entries are $\pm 1$. Therefore if $k\in J_\ell$, then $\psi$ sends $x_k$ to $\pm x_i$ for some $i \in J_\ell$.
\end{proof}

\begin{corollary} \label{coro:automorphism_of_trivial_case}
Any graded ring automorphism $\psi$ on $H^\ast(\prod_{i=1}^m \CP^{n_i})$ is induced by a self-diffeomorphism $g$ on $\prod_{i=1}^m \CP^{n_i}$, i.e., $g^\ast = \psi$.
\end{corollary}
\begin{proof}
It is clear that any automorphism $\psi$ on $\Z[x_1, \ldots, x_m] / <x_i^{N_i+1} \colon i=1, \ldots, m>$ of the form $\psi(x_i) = \pm x_{\sigma(i)}$ for some permutation $\sigma$ such that $n_i = n_{\sigma(i)}$ is actually induced by a self-diffeomorphism on $\prod_{i=1}^m \CP^{n_i}$. Hence the corollary follows from Lemma~\ref{lemma:lemma1}.
\end{proof}

We are now ready to prove Theorem~\ref{theorem:rigidity of product of projective space}.
Let $\varphi : H^\ast(M) \rightarrow H^\ast(\prod_{i=1}^m \C P^{n_i})$ be a graded ring isomorphism. By \cite[Theorem 1.1]{Ch-Ma-Su-2010}, there is a diffeomorphism $h:M \rightarrow \prod_{i=1}^m \C P^{n_i}$. Then we have
    \[
    \xymatrix{ H^\ast(M) \ar[r]^-{\varphi}& H^\ast(\prod_{i=1}^m \C P^{n_i}) \\
        H^\ast(\prod_{i=1}^m \C P^{n_i}) \ar[u]^{h^\ast} \ar[ur]_{\varphi \circ h^\ast} &
    }.
    \]
By Corollary~\ref{coro:automorphism_of_trivial_case}, there is a diffeomorphism
$$
    g: \prod_{i=1}^m \C P^{n_i} \rightarrow \prod_{i=1}^m \C P^{n_i}
$$
such that $g^\ast = \varphi \circ h^\ast$. Then $\varphi = (\varphi \circ h^\ast)\circ (h^\ast)^{-1} = g^\ast \circ (h^{-1})^\ast = (h^{-1} \circ g)^\ast = f^\ast$, where
$$
    f:= h^{-1} \circ g : \prod_{i=1}^m \C P^{n_i} \rightarrow M
$$
is a diffeomorphism such that $f^\ast = \varphi$.

\bigskip
\bibliographystyle{amsplain}
\providecommand{\bysame}{\leavevmode\hbox to3em{\hrulefill}\thinspace}
\providecommand{\MR}{\relax\ifhmode\unskip\space\fi MR }
\providecommand{\MRhref}[2]{%
  \href{http://www.ams.org/mathscinet-getitem?mr=#1}{#2}
}
\providecommand{\href}[2]{#2}

\end{document}